\documentclass[10pt]{amsart}
\usepackage{amstext,amsfonts,amssymb,amscd,amsbsy,amsmath}
\usepackage{ifthen}
\usepackage{amsthm}
\usepackage{latexsym}
\usepackage{enumerate}
\usepackage[all]{xy}

\begin{document}

\newcommand{\V}{{\cal V}}      % cal in math mode
\renewcommand{\O}{{\cal O}}
\newcommand{\LL}{\cal L}
\newcommand{\Ext}{\hbox{{\rm Ext}}}
\newcommand{\Tor}{\hbox{Tor}}
\newcommand{\Hom}{\hbox{Hom}}
\newcommand{\Proj}{\hbox{Proj}}
\newcommand{\GrMod}{\hbox{GrMod}}
\newcommand{\grmod}{\hbox{gr-mod}}
\newcommand{\tors}{\hbox{tors}}
\newcommand{\rank}{\hbox{rank}}
\newcommand{\End}{\hbox{End}}
\newcommand{\GKdim}{\hbox{GKdim}}
\newcommand{\im}{\hbox{im}}
\renewcommand{\ker}{\hbox{ker}}
\newcommand{\isom}{\cong}
\newcommand{\lk}{\text{link}}
\newcommand{\soc}{\text{soc}}

\newcommand{\lonto}{{\protect \longrightarrow\!\!\!\!\!\!\!\!\longrightarrow}}
%\newarrow{SquigTo}{$\sim$}{$\sim$}{$\sim$}{$\sim$}{>}
%\newarrow{dumb} {}-{}->

%\newcommand{\g}{\frak g}
\newcommand{\m}{{\mu}}
\newcommand{\gl}{{\frak g}{\frak l}}
\newcommand{\ssl}{{\frak s}{\frak l}}
\newcommand{\I}{\mathfrak{I}}
\newcommand{\A}{\mathfrak{A}}
\newcommand{\G}{\mathfrak{G}}

\newcommand{\ds}{\displaystyle}
\newcommand{\s}{\sigma}
\renewcommand{\l}{\lambda}
\renewcommand{\a}{\alpha}
\renewcommand{\b}{\beta}
\newcommand{\g}{\gamma}
\newcommand{\z}{\zeta}
\newcommand{\e}{\varepsilon}
\newcommand{\D}{\Delta}
\renewcommand{\d}{\delta}
\newcommand{\p}{\rho}
\renewcommand{\t}{\tau}

\newcommand{\C}{{\mathbb C}}
\newcommand{\N}{{\mathbb N}}
\newcommand{\Z}{{\mathbb Z}}
\newcommand{\ZZ}{{\mathbb Z}}
\newcommand{\K}{{\mathcal K}}
\newcommand{\F}{{\mathcal F}}
\newcommand{\cS}{\mathcal S}

\newcommand{\rowxy}{(x\ y)}
\newcommand{\colxy}{ \left({\begin{array}{c} x \\ y \end{array}}\right)}
\newcommand{\scolxy}{\left({\begin{smallmatrix} x \\ y
\end{smallmatrix}}\right)}

\renewcommand{\P}{{\Bbb P}}

\newcommand{\la}{\langle}
\newcommand{\ra}{\rangle}
\newcommand{\tensor}{\otimes}

\newtheorem{thm}{Theorem}[section]
\newtheorem{lemma}[thm]{Lemma}
\newtheorem{cor}[thm]{Corollary}
\newtheorem{prop}[thm]{Proposition}

\theoremstyle{definition}
\newtheorem{defn}[thm]{Definition}
\newtheorem{notn}[thm]{Notation}
\newtheorem{ex}[thm]{Example}
\newtheorem{rmk}[thm]{Remark}
\newtheorem{rmks}[thm]{Remarks}
\newtheorem{note}[thm]{Note}
\newtheorem{example}[thm]{Example}
\newtheorem{problem}[thm]{Problem}
\newtheorem{ques}[thm]{Question}
\newtheorem{conj}[thm]{Conjecture}
\newtheorem{thingy}[thm]{}
\newtheorem{const}[thm]{Construction}

\newcommand{\onto}{{\protect \rightarrow\!\!\!\!\!\rightarrow}}
\newcommand{\donto}{\put(0,-2){$|$}\put(-1.3,-12){$\downarrow$}{\put(-1.3,-14.5) {$\downarrow$}}}

\newcounter{letter}
\renewcommand{\theletter}{\rom{(}\alph{letter}\rom{)}}

\newenvironment{lcase}{\begin{list}{~~~~\theletter} {\usecounter{letter}
\setlength{\labelwidth4ex}{\leftmargin6ex}}}{\end{list}}

\newcounter{rnum}
\renewcommand{\thernum}{\rom{(}\roman{rnum}\rom{)}}

\newenvironment{lnum}{\begin{list}{~~~~\thernum}{\usecounter{rnum}
\setlength{\labelwidth4ex}{\leftmargin6ex}}}{\end{list}}

\title{Finiteness Conditions on the Yoneda Algebra of a Monomial Algebra}

\subjclass[2010]{Primary: } 


\author[  Conner, Kirkman, Kuzmanovich, Moore ]{ }
\maketitle

\begin{center}

\vskip-.2in Andrew Conner \\
Ellen Kirkman\\
James Kuzmanovich\\
W. Frank Moore\\
\bigskip

Department of Mathematics\\
Wake Forest University\\
Winston-Salem, NC 27109\\
\bigskip

% Edward L. Green \\
%\bigskip
%
%
%Department of Mathematics\\ Virginia Tech\\
%Blacksburg, VA 24061
%\\ \ \\

\end{center}

\setcounter{page}{1}

\thispagestyle{empty}

\vspace{0.2in}

\begin{abstract}
Let $A$ be a connected graded noncommutative monomial algebra.
We associate to $A$ a finite graph $\Gamma(A)$ called the CPS
graph of $A$.  Finiteness properties of the Yoneda algebra
$\Ext_A(k,k)$ including Noetherianity, finite GK dimension, and
finite generation are characterized in terms of $\Gamma(A)$. We show
these properties, notably finite generation, can be checked by means
of a terminating algorithm. 
\end{abstract}

\bigskip

\section{Introduction}

Complete intersections are a well-studied class of commutative
algebras, yet there is not an agreed upon notion of complete intersection
in the case of noncommutative algebras.  From the point of view of noncommutative
algebraic geometry, such a generalization should be homological.  A
starting point for a homological definition of complete intersection
is found in the results of Gulliksen \cite{Gull1,Gull2}, F\'{e}lix-Thomas \cite{FT} and
F\'{e}lix-Halperin-Thomas \cite{FHT},
which state for a graded Noetherian commutative $k$-algebra, the
following properties are equivalent:
  
\begin{enumerate}[(i)]
\item $A$ is a graded complete intersection
\item \label{noeth} $\Ext_A(k,k)$ is a Noetherian $k$-algebra
\item \label{finGK} $\Ext_A(k,k)$ has finite Gelfand-Kirillov (GK) dimension.
\end{enumerate}

However, conditions \eqref{noeth} and \eqref{finGK} are not equivalent
for graded Noetherian $k$-algebras, in fact, not even for algebras with monomial
relations.  Since such algebras are a tractable class of algebras with
a well-understood projective resolution of the trivial module (see, for example \cite{An},\cite{CS}),
their Yoneda algebras are computable, though often complex.
This paper concerns the study of conditions \eqref{noeth} and 
\eqref{finGK} as well as the finite generation of $\Ext_A(k,k)$ when $A$ is a
connected graded noncommutative $k$-algebra with finitely many monomial relations.
%In discussing our results, we will call such an algebra a monomial algebra.

To a monomial algebra $A$, we associate a finite directed graph
$\Gamma(A)$ which we call the \emph{CPS graph} of $A$. See Construction
\ref{CPSdef} for the definition of $\Gamma(A)$. Our first result
concerns the Gelfand-Kirillov dimension of the Yoneda algebra $E(A) = \Ext_A(k,k)$.

\begin{thm}[Corollary \ref{GKfacts}]
\label{intro1}
Let $A$ be a monomial $k$-algebra. If no pair of distinct circuits in
$\Gamma(A)$ have a common vertex, then $\GKdim(E(A))$ is the maximal
number of distinct circuits contained in any walk. Otherwise, $\GKdim(E(A))=\infty$.
\end{thm}

Given any connected graded $k$-algebra $B$ one can use a
noncommutative Gr\"{o}bner basis to associate to $B$ a monomial
algebra $B'$ with the property $\GKdim(E(B))\le \GKdim(E(B'))$. Thus,
Theorem \ref{intro1} can provide an easily calculated upper bound on
$\GKdim(E(B))$, though this bound is not always finite when $\GKdim E(B)$ is. See Remark \ref{gb} below.

The Yoneda product on $E(A)$ can also be described combinatorially in
terms of walks in the graph $\Gamma(A)$. Up to a notion of equivalence
described in Section \ref{CPSgraph}, all nonzero Yoneda products are
compositions of admissible walks in $\Gamma(A)$. Using this
description of the Yoneda product, we are able to characterize finite
generation and the Noetherian property in $E(A)$. See Sections 
\ref{CPSgraph} and \ref{multStructure} for definitions of terminology and notation.

\begin{thm}
\label{intro2}
Let $A$ be a monomial $k$-algebra. 
\begin{enumerate}
\item \label {introMainTheorem} (Theorem \ref{mainTheorem}) $E(A)$ is
  finitely generated if and only if for every infinite anchored walk $p$ in
  $\Gamma(A)$, $\widetilde{p}$ contains a dense edge
  or two admissible edges of opposite parity.
\item (Theorem \ref{Noeth}) $E(A)$ is left (resp. right) Noetherian if
  and only if every vertex of $\Gamma(A)$ lying on an oriented circuit
  has out-degree (resp. in-degree) one and every edge of every
  oriented circuit is admissible.
\end{enumerate}
\end{thm}

The second statement extends a theorem of Green et.\ al.\ \cite{GSSZ}
who characterized Noetherianity of $E(A)$ in terms of the Ufnarovski
relation graph of $A$ in the case where $A$ is quadratic.

Theorem \ref{intro2}\eqref{introMainTheorem} describes an infinite set
of criteria to be satisfied for $E(A)$ to be finitely
generated. Whether finite generation of $E(A)$ can be determined by
finitely many criteria is a problem of recent interest. Working in the
more general context of monomial factor algebras of quiver path algebras, Green and
Zacharia \cite{GZ} describe a (potentially infinite) process by which
finite generation of the Yoneda algebra can be checked. Further
progress was made by Davis \cite{Davis} and Cone \cite{ConeThesis} who
showed finite generation can be determined by finitely many criteria
when the given quiver is a cycle or an ``in-spoked'' cycle. In Section
4 we show the same can be said in our situation; that is, when
the quiver consists of a single vertex and finitely many loops.

\begin{thm}[Theorem \ref{finiteCheck}]
\label{intro3}
Let $A$ be a monomial $k$-algebra with $gl.dim\ A = \infty$. Let $N$
be the smallest even integer greater than or equal to $2\mathcal E^2+\mathcal E+1$ where
$\mathcal E$ is the number of edges in $\Gamma(A)$. The Yoneda algebra $E(A)$
is finitely generated if and only if every anchored walk of length $N$ or
$N+1$ is decomposable.
\end{thm}

In our experience, determining if $E(A)$ is finitely generated when
$E(A)$ has infinite Gelfand-Kirillov dimension can be a difficult problem, and we were unable to obtain an efficient bound in Theorem
\ref{intro3}.  However, the case $\GKdim(E(A))<\infty$ is much
simpler. We describe a recursive algorithm for determining finite
generation in that case in Section \ref{algorithm}.

We sincerely thank Ed Green for the helpful conversations and illuminating examples he provided in the course of this project.

\section{The CPS graph}
\label{CPSgraph}

In \cite{PhanThesis}, C.\ Phan associated a weighted digraph to any monomial graded algebra $A$. One important feature of Phan's graph is that a $k$-basis for $E(A)$ is represented by certain directed paths. After establishing some notation, we recall the unweighted version of Phan's graph - which we call the \emph{CPS graph} of $A$ - and we record a description of a minimal graded projective resolution of $_Ak$ (due to Cassidy and Shelton) in terms of this graph. We also prove several combinatorial facts about the CPS graph needed later.

Let $k$ be a field.  Throughout this paper we use the phrase \emph{graded $k$-algebra} or just \emph{$k$-algebra} to mean a connected, $\N$-graded, locally finite-dimensional $k$-algebra which is finitely generated in degree 1.  If $A$ is a graded $k$-algebra, we use the term \emph{(left or right) ideal} to mean a graded (left or right) ideal of $A$ generated by homogeneous elements of degree at least 2, unless otherwise indicated. The augmentation ideal is $A_+=\bigoplus_{i\ge 1} A_i$. We abuse notation and use $k$ (or $_Ak$ or $k_A$) to denote the trivial graded $A$-module $A/A_+$. The bigraded Yoneda algeba of $A$ is  the $k$-algebra $E(A) = \bigoplus_{i,j\ge 0} E^{i,j}(A) = \bigoplus_{i,j\ge 0} \Ext^{i,j}_A(k,k)$. (Here $i$ denotes the cohomology degree and $j$ denotes the internal degree inherited from the grading on $A$.) Let $E^p(A)=\bigoplus_{q} E^{p,q}(A)$.

Let $s\in\N$ and let $V=\text{span}_k\{x_1,\ldots,x_s\}$. We denote the tensor algebra on $V$ by $T(V)$. The tensor algebra is a graded $k$-algebra, graded by tensor degree. We denote the tensor degree of a homogeneous element $w\in T(V)$ by $\deg w$. By a \emph{monomial} in $T(V)$ we mean a pure tensor with coefficient 1. We consider $1_{T(V)}$ a monomial. By a \emph{monomial algebra}, we mean an algebra of the form $A=T(V)/I$ where $I$ is an ideal of $T(V)$ generated by finitely many monomials.  Such an algebra $A$ is a graded $k$-algebra with the grading inherited from the tensor grading on $T(V)$.

Let $M$ be the set of monomials in $T(V)$. Multiplication in $T(V)$ induces the structure of a monoid on $M$. Let $I=\la w_1,\ldots w_r\ra$ be an ideal in $T(V)$. We assume the $w_i$ form a minimal set of monomial generators for $I$ and we let $d_i=\deg w_i$ be the tensor degree of $w_i$ for each $i$. Recall that we assume every $d_i\ge 2$. Let $A=T(V)/I$ and let $\pi:T(V)\rightarrow A$ be the natural surjection.

\begin{const}[CPS graph]
\label{CPSdef}
Suppose $m, w\in M-I$ and $w\tensor m\in I$. Let $L(w,m)=w'$ where $w=w''\tensor w'$ for $w', w''\in M$ and $w'$ is minimal such that $w'\tensor m\in I$. For $m\in M-I$ define
$$\A_m = \{ w\in M - I : w\tensor m\in I \text{ and } L(w,m)=w\}$$
Then the images of elements of $\A_m$ in $A$ generate the left annihilator of $\pi(m)$.

Let $\G_0 = \{x_1, \ldots x_s\}$ and for $i\ge 1$ let $\G_i = \bigcup_{w\in \G_{i-1}} \A_w$. Finally, let $\G=\bigcup_{i\ge 0} \G_i$.
Define the \emph{CPS graph} of $A$ to be the directed graph $\Gamma(A)$ with vertex set $\G$, and edges $m_1\rightarrow m_2$ whenever $m_2\in\A_{m_1}$.
\end{const}
We note the graph $\Gamma(A)$ is finite. The graph may have loops and parallel edges with opposite orientation, but it has no parallel edges with the same orientation. 

\begin{ex}
\label{ex1}
Let $A=k\la a,b,c,d\ra/\la abc, cdab\ra$ and $B=A/\la bcda\ra$ The graphs $\Gamma(A)$ and $\Gamma(B)$ are shown below. 
$$
\xymatrix{
& \Gamma(A) &\\
c\ar@{->}[r] & ab\ar@/^/[r] & cd\ar@/^/@{->}[l]\\
b\ar@{->}[r] & cda\ar@{->}[u] &\\
}\qquad 
\xymatrix{
& \Gamma(B) &\\
c\ar@{->}[r] & ab\ar@/^/[r] & cd\ar@/^/@{->}[l]\\
b\ar@/^/[r] & cda\ar@/^/[l] &\\
a\ar@/^/[r] & bcd\ar@/^/[l] &\\
}
$$
\end{ex}

\begin{rmk}
\label{G0edge}
An obvious, but extremely important feature of the CPS graph is that there is a directed edge $m_1\rightarrow m_2$ with $m_1\in\G_0$ if and only if $m_2\tensor m_1$ is a minimal generator of $I$. As illustrated by Proposition \ref{pathBasis} below, this correspondence parallels the standard identification of $\Ext^2_A(k,k)$ with the graded dual of the space $I/(V\tensor I+I\tensor V)$.
\end{rmk}
 
If the defining relations of a monomial algebra $A$ are quadratic, $A$ is Koszul \cite{PP}. In that case, $\Gamma(A)$ is Ufnarovski's ``relation graph'' \cite{U} for the Koszul dual algebra $A^!$. We also note that because we consider only minimal left annihilators, the CPS graph $\Gamma(A)$ is quite different from the notion of ``zero-divisor graph'' studied recently in \cite{Akbari}. 
 
We adopt some standard graph-theoretic terminology. By a \emph{walk} we mean a finite or infinite sequence $v_0v_1v_2\cdots$ of vertices where  $v_i\rightarrow v_{i+1}$ is a directed edge for all $0\le i<n$. If $v_0 v_1 \cdots v_n$ is a finite walk, we say the walk has \emph{length $n$}. A walk is called a \emph{path} if it contains no repeated vertices. We will not need to distinguish walks which repeat vertices but not edges. By a \emph{closed walk of length $n$} we mean a walk of length $n$ such that $v_n=v_0$. A \emph{circuit of length} $n$ is a closed walk of length $n$ such that $v_0,\ldots,v_{n-1}$ are distinct. In the context of a weighted digraph, we abuse this terminology slightly and use ``walk,'' ``path,'' and ``circuit'' to refer to sequences of vertices in the underlying unweighted graph. If $p$ and $q$ are walks of length $n$ and $m$ respectively, we say $p$ \emph{extends} $q$ or $q$ is a \emph{prefix} of $p$ and write $q\vdash p$ if $n\ge m$ and $p_i=q_i$ for all $0\le i\le m$. 
 
In \cite[\S 5]{CS}, Cassidy and Shelton give a combinatorial description of a minimal graded projective left $A$-module resolution $P_{\bullet}$ of $_Ak$ in terms of monomial matrices. We briefly recount their resolution here, indexing the bases of each graded projective module by certain walks in $\Gamma(A)$.

Let $\mathcal W_n$ denote the set of all walks $w$ of length $n$ in $\Gamma(A)$ such that $w_0\in\G_0$. For each $w\in\mathcal W_n$, let $d_w=\sum_{i=0}^n \deg w_i$ where $\deg w_i$ denotes the tensor degree of the monomial $w_i$. Let $A(-d_w)$ be the graded free left $A$-module of rank 1 with grading shift $A(-d_w)_p = A_{p-d_w}$. Choose a basis for $A(-d_w)$ and denote this element by $e_w$. Let $P_0=A$ be the graded free module with fixed basis element $e_{\emptyset}$ and for $j>0$, let 
$$P_j=\bigoplus_{w\in\mathcal W_{j-1}} A(-d_w)$$
Define $d_j:P_j\rightarrow P_{j-1}$ on the $A$-basis $\{e_w : w\in\mathcal W_{j-1}\}$ by setting $d_j(e_w) = \pi(w_{j-1}) e_{\bar w}$ where $\bar w = w_0\cdots w_{j-2}$ if $j\ge 2$ and $\bar w=\emptyset$ if $j=1$. Extend $d_j$ $A$-linearly to all of $P_j$. Since $w_{j-1}\rightarrow w_{j}$ is an edge in $\Gamma(A)$ only if $w_j\in\A_{w_{j-1}}$, it is clear that $d_jd_{j+1}=0$ for $j\ge 0$.
The following lemma is a straightforward consequence of the definition of $\Gamma(A)$.

\begin{lemma}
\label{graphComplex}
The complex $(P_{\bullet},d_{\bullet})$ described above is a minimal graded projective resolution of $_Ak$.
\end{lemma}
 
Moreover, the bases for the $P_j$ can be ordered so the matrices of the $d_j$ with respect to the ordered bases are precisely the monomial matrices described in \cite{CS}. The next fact follows immediately from Lemma \ref{graphComplex}.
 
 \begin{prop}
 \label{pathBasis}
 Let $A$ be a monomial $k$-algebra and $i\in\N$. Then the graded duals $\{\e_w\}$ of the basis elements $\{e_w\}$ where $w$ is a walk of length $i$ in $\Gamma(A)$ with $w_0\in\G_0$ form a $k$-basis for $\Ext_A^{i+1}(k,k)$.
 \end{prop}
 
 We make extensive use of this basis throughout the paper. For ease of exposition, we make the following defintion.
 
 \begin{defn}
 A walk $w$ in $\Gamma(A)$ is called \emph{anchored} if $w_0\in\G_0$. 
 \end{defn}
 
 \begin{rmk} Anchored walks of length $i$ in $\Gamma(A)$ correspond to the sets $\Gamma_i$ described in \cite {GZ}.  \end{rmk}

 Several properties of $A$ and $E(A)$ are immediate from Proposition \ref{pathBasis}. We denote the Gelfand-Kirillov dimension of a $k$-algebra $A$ by $\GKdim(A)$.
 
 \begin{cor}
 \label{GKfacts}\  
 \begin{enumerate}
 \item If $\Gamma(A)$ contains no circuit, then $\text{gl.dim}(A)$ is equal to the length of the longest path in $\Gamma(A)$. Otherwise, $\text{gl.dim}(A)=\infty$.
 \item $\GKdim(E(A))=\infty$ if and only if $\Gamma(A)$ contains distinct circuits with a common vertex.
 \item If no pair of distinct circuits in $\Gamma(A)$ have a vertex in common, then $\GKdim(E(A))$ is the maximal number of circuits contained in any walk (ignoring multiplicity). 
 \item The Hilbert series of $E(A)$ is a rational function.
 \end{enumerate}
 \end{cor}
 
 \begin{proof}
 (1) is clear. (2), (3), and (4) are standard (see \cite{U}).
 
 \end{proof}

\begin{rmk}
\label{gb}
To any connected graded $k$-algebra $B=T(V)/J$ (we do not assume $J$ is generated by monomials) one can associate a monomial algebra in the usual way: Choose an ordered basis of $V$ and induce a total ordering the monoid $M$ via degree-lexicographic order. Let $\mathcal F$ be a noncommutative Gr\"{o}bner basis of $J$ with respect to this ordering.  Let $ht(\mathcal F)$ be the set of high terms of elements of $\mathcal F$ and let $B'=T(V)/\la ht(\mathcal F)\ra$. Let $$P_B(y,z) = \sum_{p,q} \dim \Ext_B^{p,q}(k,k)y^pz^q$$ denote the Poincare series of $B$. From the well-known coefficientwise inequality $P_B(y,z)\le P_{B'}(y,z)$ (see Lemma 3.4 of \cite{An}) we can deduce $\GKdim(E(B))\le \GKdim(E(B'))$. Equality holds in the important case where the Gr\"{o}bner basis for $J$ consists of homogeneous polynomials of the same degree (see Corollary 4.6 of \cite{JoW}). Thus Corollary \ref{GKfacts} can sometimes provide an easily calculated upper bound on $\GKdim(E(B))$. For further examples, see Section \ref{examples}.
\end{rmk}

It is also interesting to note that $E(A)$ has either exponential or
polynomial growth - this is the case for commutative $k$-algebras
(see \cite{Avramov,Gull1,Gull2}). Observe $\GKdim(E(A))=\GKdim(E(B))=1$
for algebras $A$ and $B$ from Example
\ref{ex1}.

Given a minimal projective resolution of $_Ak$, one can compute the Yoneda product of classes $\e_1$ and $\e_2$ in $E(A)$ by lifting a representative  of $\e_2$ through the resolution to the appropriate cohomology degree and composing with a representative of $\e_1$. For a monomial algebra $A$, we wish to describe the Yoneda product combinatorially in terms of walks in the graph $\Gamma(A)$. To do this, we introduce a notion of walk equivalence as a combinatorial analog of lifting a representative through a projective resolution.

 We call two walks $p=p_0\cdots p_n$ and $q=q_0\cdots q_m$ in a CPS graph $\Gamma(A)$ \emph{equivalent} if $m=n$ and 
 $$p_n\tensor p_{n-1}\tensor\cdots \tensor p_0=q_m\tensor q_{m-1}\tensor\cdots\tensor q_0$$ as elements of $M$. If $p$ and $q$ are equivalent, we write $p\sim q$. It is clear that $\sim$ is an equivalence relation on walks in $\Gamma(A)$.
 
 \begin{lemma}
 \label{pathEquivalence}
 Let $\Gamma(A)$ be a CPS graph, and let $p$ and $q$ be equivalent walks of length $n>0$ in $\Gamma(A)$. Then
 \begin{enumerate}
 \item the prefix walks $p_0\cdots p_{2k+1}$ and $q_0\cdots q_{2k+1}$ are equivalent for all $0\le k\le \lfloor \frac{n}{2}\rfloor $.
 \item if $n$ is even, then $p_n=q_n$.
 \item we have $p_{2k+1}\tensor p_{2k}=q_{2k+1}\tensor q_{2k}$ for all $0\le k\le  \lfloor \frac{n}{2}\rfloor $.
  \item if $\deg(p_0)\ge \deg(q_0)$, then 
 \begin{align*}
 \deg(p_i) &\ge \deg(q_i) \text{ if $0< i\le n$ is even }\\ 
 \deg(q_i) &\ge \deg(p_i) \text{ if  $0<i\le n$ is odd}
 \end{align*}
 \item the walk $q$ is unique if it is anchored.
   \end{enumerate}
 \end{lemma} 
 
 \begin{proof}  To prove (1),
 we induct on $k$. Let $k=0$. By switching the variables $p$ and $q$ if necessary, there is no loss of generality in assuming $\deg(p_0)\ge \deg(q_0)$. Since $$p_n\tensor\cdots\tensor p_0=q_n\tensor\cdots\tensor q_0$$  there exists a unique monomial $m\in M-I$ such that $m\tensor q_0=p_0$. We have  $q_1\tensor q_0\in I$, $m\tensor q_0\notin I$, and 
 $$p_n\tensor\cdots\tensor p_1\tensor m\tensor q_0=q_n\tensor\cdots\tensor q_1\tensor q_0$$
 so there is a unique monomial $m'\in M-I$ such that $q_1=m'\tensor m$. Now, $m'\tensor p_0=m'\tensor m\tensor q_0=q_1\tensor q_0\in I$ and $L(q_1,q_0)=q_1$, so $L(m', p_0)=m'$ and $m'\in \A_{p_0}$. Thus
 \begin{align*}
 p_n\tensor\cdots\tensor p_1\tensor p_0 &= q_n\tensor\cdots\tensor q_1\tensor q_0\\
 &=q_n\tensor \cdots \tensor m' \tensor p_0
 \end{align*}
  and $p_1, m'\in\A_{p_0}$, so $p_1=m'$. Hence 
  $$p_1\tensor p_0=m'\tensor p_0=q_1\tensor q_0$$ 
  as desired. For the induction step, assume 
  $$p_{2k+1}\tensor\cdots\tensor p_0=q_{2k+1}\tensor\cdots\tensor q_0$$ and $\deg(p_{2k+2})\ge\deg(q_{2k+2})$ and proceed as in the base case. This completes the proof of (1).
 
  Statements (2) and (3) follow immediately from (1).
 
  We consider statement (4). In light of (2) and (3), it suffices to prove $\deg(q_i)\ge\deg(p_i)$ for $0<i\le n$ odd. Since $q_1\tensor q_0=p_1\tensor p_0$ and $\deg(p_0)\ge\deg(q_0)$, it is clear that $\deg(q_1)\ge \deg(p_1)$. Thus the result holds for $n\le 2$. Assume $n>2$ and for $0<i<n-1$ odd $\deg(q_i)\ge \deg(p_i)$. Since $q_i\tensor q_{i-1}=p_i\tensor p_{i-1}$, there exists $m\in M$ such that $q_i=p_i\tensor m$. Suppose toward contradiction that $\deg(p_{i+1})<\deg(q_{i+1})$. Since $q_{i+2}\tensor q_{i+1}=p_{i+2}\tensor p_{i+1}$, there exists $m'\in M$, $\deg(m')>0$ such that $q_{i+1}=m'\tensor p_{i+1}$. Since $p_{i+1}\tensor p_i\in I$, we have $p_{i+1}\tensor q_i=p_{i+1}\tensor p_i\tensor m\in I$. The fact that $\deg(m')>0$ contradicts the assumption that $L(q_{i+1},q_i)=q_{i+1}$. So $\deg(q_{i+1})\le \deg(p_{i+1})$ and hence $\deg(q_{i+2})\ge \deg(p_{i+2})$. Statement (4) now follows by induction.
 
  To prove (5), suppose $q'$ is another walk such that $p\sim q'$ and $q'_0\in \G_0$. Then $q\sim q'$ and $q_1\tensor q_0=q'_1\tensor q'_0$. Since $\G_0$ consists solely of degree 1 monomials, $q_0=q_0'$ so $q_1=q_1'$. 
 
 Suppose inductively that $q_i=q_i'$ for all $0\le i\le 2k+1<n$. If $n=2k+2$, the induction hypothesis and the definition of equivalence imply $q_{2k+2}=q_{2k+2}'$. 
 
 If $n>2k+2$, $q_{2k+3}\tensor q_{2k+2}=q_{2k+3}'\tensor q_{2k+2}'$. By switching the variables $q$ and $q'$ if necessary, we can assume $\deg(q_{2k+2})\ge \deg(q_{2k+2}')$,  so $q_{2k+2}=m\tensor q'_{2k+2}$ for some $m\in M$. But $L(q_{2k+2},q_{2k+1})=q_{2k+2}$ and
 $$q'_{2k+2}\tensor q_{2k+1}=q'_{2k+2}\tensor q'_{2k+1}\in I$$ by the induction hypothesis. Thus $m=1$, $q_{2k+2}=q_{2k+2}'$, and hence $q_{2k+3}=q_{2k+3}'$. Statement (5) now follows by induction.

 \end{proof}
 
 Since finite anchored walks in $\Gamma(A)$ enumerate a $k$-basis for $E(A)$, we make the following definition.
 
\begin{defn}
A finite walk in $\Gamma(A)$ is called \emph{admissible} if it is equivalent to an anchored walk.
\end{defn}
  
By Lemma \ref{pathEquivalence}(5), every admissible walk is equivalent to a unique anchored walk. In Example \ref{ex1}, edge $ab\rightarrow cd$ is an admissible walk of length 1 in both $\Gamma(A)$ and $\Gamma(B)$, but edge $cd\rightarrow ab$ is admissible in neither graph. In Phan's original weighted digraph, the edge weighting distinguished admissible edges from their counterparts. That distinction is too coarse for our purposes, but the importance of admissible edges seems evident from the following useful facts about admissible walks.
 
\begin{prop}
\label{pathExtensions}
Let $\Gamma(A)$ be a CPS graph, and let $p$ be an admissible walk of length $n$ in $\Gamma(A)$. Let $q$ be a walk of length $s$ such that $q$ extends $p$. If either $n$ or $s-n$ is even, then $q$ is admissible.
\end{prop}
 
\begin{proof}
An admissible walk of length 0 consists of a single vertex in $\G_0$, so the statement is trivial if $n=0$. The statement is also trivial if $s=n$. So assume $n>0$, $s-n>0$, and let $r$ be a path in $\Gamma(A)$ such that $p\sim r$ and $r_0\in \G_0$. 

If $n$ is even, then $r_n=p_n$ by Lemma \ref{pathEquivalence}(2). It follows immediately that the path $r'=r_0'\cdots r_{s}'$ given by $r_i'=r_i$ for $0\le i\le n$ and $r'_i=q_{i}$ for $n+1\le i\le s$ is equivalent to $q$ and has $r'_0\in \G_0$. 

Suppose $n$ and $s$ are odd. If $r_n=p_n$, we can proceed as above, so assume $r_n\neq p_n$. By Lemma \ref{pathEquivalence}(3) and (4), there exists a monomial $m\in M$, $\deg(m)>0$ such that  $r_n=p_n\tensor m=q_n\tensor m$. Thus $q_{n+1}\tensor r_n\in I$. Put $r_{n+1}=L(q_{n+1},r_n)$ and let $m'\in M$ such that $q_{n+1}=m'\tensor r_{n+1}$. Now, $$q_{n+2}\tensor m'\tensor r_{n+1}=q_{n+2}\tensor q_{n+1}\in I$$ Since $q_{n+1}\notin I$ and $L(q_{n+2}, q_{n+1})=q_{n+2}$, it follows that $q_{n+2}\tensor m'\in \A_{r_{n+1}}$. Put $r_{n+2}=q_{n+2}\tensor m'$. Then by construction, $r_0\cdots r_{n+2}$ is a well-defined walk equivalent to $p'=q_0\cdots q_{n+2}$ and $r_0\in \G_0$. Thus $p'$ is an admissible walk of length $n+2$ and $q$ is an extension of $p'$ of length $s$. The result now follows by induction on $s-n$.

\end{proof}

We show in the next section that $E(A)$ is finitely generated if $\Gamma(A)$ has ``enough'' admissible walks. To make this more precise, we make the following definition.

\begin{defn}
Let $p$ be an infinite walk in $\Gamma(A)$ and let $e=p_ip_{i+1}$ be an admissible edge in $p$. We call $e$ \emph{dense} in $p$ if $e$ has an admissible even-length extension in $p$.
\end{defn}

In Example \ref{ex1}, $c\rightarrow ab\rightarrow cd\rightarrow ab\rightarrow cd\cdots$ is the only infinite anchored walk in $\Gamma(A)$.
The admissible edge $ab\rightarrow cd$ is dense in this walk since
$$ab\rightarrow cd\rightarrow ab\qquad \sim\qquad b\rightarrow cda\rightarrow ab$$
However, the edge $ab\rightarrow cd$ is not dense in the same walk in $\Gamma(B)$. The equivalent anchored walks corresponding to odd-length extensions of $ab\rightarrow cd$ begin at vertex $b$ and end at either $b$ or $cda$. It follows that no even length extension of $ab\rightarrow cd$ is admissible because condition (2) of Lemma \ref{pathEquivalence} cannot be satisfied.

An admissible edge $e$ may belong to many infinite walks. The edge $e$ may be dense in some infinite walks, but not others. Furthermore, $e$ may not be dense in an infinite walk $w$, but $w$ may contain some other dense edge. See Example \ref{infGK}.

The following criterion for establishing density is immediate from Lemma \ref{pathEquivalence}(2). 

\begin{lemma}
\label{density}
Let $w$ be a (possibly infinite) walk in $\Gamma(A)$ and let $e=w_iw_{i+1}$ be an admissible edge in $w$. Let $q$ be any odd-length extension of $e$ in $w$ and let $q'$ be the unique anchored walk equivalent to $q$. Then every even-length extension of $q$ in $w$ is admissible if and only if $q'_t=w_{i+t}$ for some even $t\ge 0$.

\end{lemma}

\section{Multiplicative Structure}
\label{multStructure}

 In this section we show certain extensions of walks in $\Gamma(A)$ correspond to Yoneda products in $E(A)$ and use the result to combinatorially characterize finite generation of $E(A)$.

 Recall that if $w$ is an anchored walk of length $n$ in $\Gamma(A)$, we denote the corresponding $A$-basis element of $P_{n+1}$ by $e_w$. We denote the graded dual of $e_w$ by $\e_w$.
 
 Fix an anchored walk $q$ of length $n$. To connect the Yoneda product in $E(A)$ to extensions of walks in $\Gamma(A)$, we explicitly construct lifts of $\e_q$ through the resolution $(P_{\bullet}, d_{\bullet})$ defined in Section \ref{CPSgraph}. We need one additional definition before describing the construction.

 \begin{defn}[\cite{CS}]
 An element $r$ in an ideal $I\subset T(V)$ is called \emph{essential} if $r$ is not in the ideal generated by $V\tensor I+I\tensor V$.
 \end{defn}
 
 We note that a monomial $r$ in a monomial ideal $I$ is essential if and only if $r$ is a minimal generator of $I$. Hence a walk $w_0w_1$ in $\Gamma(A)$ is admissible if and only if $w_1\tensor w_0$ is a minimal generator of $I$.
 
 For $i\ge 0$, we define $A$-module maps $f_i$ such that the following diagram commutes.

\begin{equation}
\label{lifts}
\xymatrix{
\cdots \ar@{->}[r] & P_{n+3}\ar@{->}[r]^{d_{n+3}}\ar@{->}[d]_{f_2} & P_{n+2}\ar@{->}[r]^{d_{n+2}}\ar@{->}[d]_{f_1} & P_{n+1}\ar@{->}[d]_{f_0}\ar@{->}[dr]^{\e_q} & \\
\cdots\ar@{->}[r] & P_{2}\ar@{->}[r]^{d_{2}} & P_{1}\ar@{->}[r]^{d_{1}} & P_{0}\ar@{->}[r] & k \\
}
\end{equation}
\bigskip

For $i\ge 0$ let $Q_{n+i}$ be the graded free submodule of $P_{n+i}$ spanned by the set $\{e_r : q\vdash r\}$ and let $Z_{n+i}$ be the complement to $Q_{n+i}$ in $P_{n+i}$. We observe that $P_{\ge n}=Q_{\ge n}\oplus Z_{\ge n}$ as complexes of graded free left $A$-modules. For all $i\ge 0$, we define $f_i(Z_{n+i+1})=0$. 

We define $f_i$ on the specified $A$-basis of $Q_{n+i+1}$ in several steps.
\begin{enumerate}
\item[(i)] Define $f_0(e_q)=e_{\emptyset}$.
\item[(ii)] For any walk $r$ such that $e_r\in Q_{n+2}$, we have $r_{n+1}=m\tensor x_j$ for a unique $m\in M$ and generator $x_j$. Define  $f_1(e_r) = \pi(m)e_{x_j}$. 
\item[(iii)] 
Suppose $d>2$ and $r$ is a walk such that $e_r\in Q_{n+d}$.
Recall the walk $r_{n+1}r_{n+2}$ is admissible if and only if $r_{n+2}\tensor r_{n+1}$ is essential. (See Remark \ref{G0edge}.) If $r_{n+2}\tensor r_{n+1}$ is not essential, we define $f_{d-1}(e_r)=0$. If $r_{n+2}\tensor r_{n+1}$ is essential, our construction depends on the partiy of $d$.

If $d$ is odd and $r_{n+1}r_{n+2}$ is admissible, the walk $r_{n+1}\cdots r_{n+d-1}$ of length $d-2$ is admissible by Proposition \ref{pathExtensions}. Let $r'=r_{0}'\cdots r_{d-2}'$ be the unique anchored walk equivalent to $r_{n+1}\cdots r_{n+d-1}$. Then $e_{r'}\in P_{d-1}$ and we define $f_{d-1}(e_r)= e_{r'}$.

If $d$ is even and $r_{n+1}r_{n+2}$ is admissible, then the length $d-3$ walk $r_{n+1}\cdots r_{n+d-2}$ is admissible. Let $r_0'\cdots r_{d-3}'$ be the equivalent anchored walk. Lemma \ref{pathEquivalence}(3) and (4) imply that there exists a unique monomial $m\in M$ such that $r_{n+d-2}\tensor m=r_{d-3}'$. Since $r_{n+d-1}\tensor r_{n+d-2}\in I$, we have $r_{n+d-1}\tensor r_{d-3}'\in I$. Put $r_{d-2}'=L(r_{n+d-1},r_{d-3}')$ and let $m'\in M$ be the unique monomial such that $r_{n+d-1}=m'\tensor r_{d-2}'$. Then $r'=r_0'\cdots r_{d-2}'$ is a well-defined anchored walk, $e_{r'}\in P_{d-1}$, and we may define $f_{d-1}(e_r) =  \pi(m')e_{r'}$.
\end{enumerate}

To summarize: For $d>0$ and $r$ a walk such that $e_r\in Q_{n+d}$, we define

$$ f_{d-1}(e_r)=
\begin{cases}
e_{\emptyset} & \text{ if } d=1\\
\pi(m)e_{x_j} & \text{ if $d=2$ and $r_{n+1}=m\tensor x_j$} \\
e_{r'} & \text{ if $r_{n+2}\tensor r_{n+1}$ is essential and $d>2$ odd}\\
\pi(m')e_{r'} &  \text{ if $r_{n+2}\tensor r_{n+1}$ is essential and $d>2$ even}\\ 
0 & \text{ else}\\
\end{cases}
$$
where $r'$ is a uniquely determined anchored walk of length $d-2$ and $r_{n+d-1}=m'\tensor r_{d-2}'$. Extending the definitions of the $f_i$ $A$-linearly, we obtain a sequence of $A$-module maps.

\begin{lemma}
With $f_i$ defined as above, the diagram (\ref{lifts}) commutes. %That is, $d_i(f_i(e)) = f_{i-1}(d_{n+i+1}(e))$ for all $i\ge 1$ and basis elements $e\in P_{n+i+1}$.
\end{lemma}

\begin{proof} Because $f_i(Z_{n+i+1})=0$ for all $i\ge 0$ and $Z_{>n}$ is a subcomplex of $P_{>n}$, it suffices to show commutativity for the complex $Q_{>n}$. We compute the first few squares explicitly.

\begin{itemize}
\item[(d=1)] The augmentation map $\epsilon : P_0\rightarrow k$ takes $e_{\emptyset}\mapsto 1$, so $\e_q=\epsilon f_0$.

\item[(d=2)] Let $r$ be a walk of length $n+1$ in $\Gamma(A)$ which extends $q$. Then 
$$f_0d_{n+2}(e_r)=f_0(\pi(r_{n+1})e_q)=\pi(r_{n+1})e_{\emptyset}$$ 
On the other hand, $r_{n+1}=m\tensor x_j$ for unique $m\in M$ and generator $x_j$ so
$$d_1f_1(e_r)=d_1(\pi(m)e_{x_j}) = \pi(m)\pi(x_j)e_{\emptyset} = \pi(r_{n+1})e_{\emptyset}$$

\item[(d=3)] Let $r$ be a walk of length $n+2$ in $\Gamma(A)$ which extends $q$. If $r_{n+2}\tensor r_{n+1}$ is essential, then $$d_{2}f_{2}(e_r)=d_{2}(e_{r'})=\pi(r'_{1})e_{\overline{r'}}$$
where $r'=r'_0r_1'$ is anchored, equivalent to $r_{n+1}r_{n+2}$, and $\overline{r'}=r_0'$. On the other hand,
$$f_{1}(d_{n+3}(e_r))=f_{1}(\pi(r_{n+2})e_{\bar{r}})=\pi(r_{n+2})\pi(m)e_{x_j}=\pi(r_{n+2}\tensor m)e_{x_j}$$
where $\bar{r}=r_0\cdots r_{n+1}$ and $r_{n+1}=m\tensor x_j$. In this case, since $r'_0\in\G_0$ and $r'_1\tensor r'_0=r_{n+2}\tensor r_{n+1}$, we have $r'_0=x_j$ and $r'_1=r_{n+2}\tensor m$ as desired. 

If $r_{n+2}\tensor r_{n+1}$ is not essential, $f_{2}(e_r)=0$ and $\pi(r_{n+2}\tensor m)=0$ since $r_{n+2}\in\A_{r_{n+1}}$.

\end{itemize}

For $d>3$, the arguments are similar to those above and are omitted. The key observation is that $r'_0\cdots r'_{d-4}$ is equivalent to $r_{n+1}\cdots r_{n+d-3}$ by Lemma \ref{pathEquivalence}(1) so $\bar{r}'=r'_0\cdots r'_{d-4}$ by uniqueness (Lemma \ref{pathEquivalence}(5)). The definitions of the $f_i$ then imply  $r_{n+d-2}=m'\tensor r'_{d-3}$ if $d$ is odd and $r_{n+d-1}=m'\tensor r'_{d-2}$ if $d$ is even, from which commutativity follows.

%Let $d>3$ be odd and let $r$ be a walk of length $n+d-1$ which extends $q$. If $r_{n+2}\tensor r_{n+1}$ is essential, then $$d_{d-1}f_{d-1}(e_r)=d_{d-1}(e_{r'})=\pi(r'_{d-2})e_{\overline{r'}}$$ 
%where $r'=r_0'\cdots r_{d-2}'$ is equivalent to $r_{n+1}\cdots r_{n+d-1}$, $r_0'\in\G_0$, and $\overline{r'}=r'_0,\ldots, r'_{d-3}$. On the other hand, 
%$$d_{n+d}(e_r)=\pi(r_{n+d-1})e_{\bar{r}}$$ where $\bar{r}=r_0\cdots r_{n+d-2}$. 
%
%If $d>3$, then $\bar{r}'$ is equivalent to $r_{n+1}\cdots r_{n+d-3}$ by Lemma \ref{pathEquivalence}(1), hence $\bar{r}'\sim \overline{r'}$ by Lemma \ref{pathEquivalence}(5). We have
%$$f_{d-2}(d_{n+d}(e_r))=f_{d-2}(\pi(r_{n+d-1})e_{\bar{r}})=\pi(r_{n+d-1})\pi(m')e_{\bar{r}'}$$ 
%where $\bar{r}'$ 
%
%$m'r'_{d-3}=r_{n+d-2}$. By Lemma \ref{pathEquivalence}, $r_{n+d-1}r_{n+d-2}=r'_{d-2}r'_{d-3}$, so $r_{n+d-1}m'=r'_{d-2}$.
%
%
%
%The calculation for $d>2$ even is similar and is omitted.

%Let $d>2$ be even and let $r$ be a walk of length $n+d-1$ which extends $q$. If $r_{n+2}r_{n+1}$ is essential, then $f_{d-1}(e_r)M_{d-1}=m'e_{r'}M_{d-1}=m'r'_{d-2}e_{\bar{r}}=r_{n+d-1}e_{\bar{r}}$ where $\bar{r}=r'_0,\ldots, r'_{d-3}$. On the other hand, $e_rM_{n+d}=r_{n+d-1}e_{\tilde{r}}$ where $\tilde{r}=r_{n+1},\ldots, r_{n+d-2}$. Since $d-2>2$, $\tilde{r}$ is admissible. It follows that $r'_0,\ldots, r'_{d-4}\sim r_{n+1},\ldots, r_{n+d-3}$. We have $f_{d-2}(e_rM_{n+d})=f_{d-2}(r_{n+d-1}e_{\tilde{r}})=r_{n+d-1}e_{\bar r}$.
%Again, we have  $f_{d-1}(e_r)=0$ and $f_{d-2}(e_{\tilde{r}})=0$ when $r_{n+2}r_{n+1}$ is not essential, so the result holds when $i$ is odd.

\end{proof}

If $\alpha, \beta\in E(A)$, we denote the Yoneda composition product by $\alpha\star\beta$. If $w$ is an admissible walk of length $m$ in $\Gamma(A)$ (not necessarily anchored), we define the symbol $\e_w$ to mean the dual basis element $\e_q$ where $q$ is the unique anchored walk equivalent to $w$ guaranteed by Lemma \ref{pathEquivalence}. The following proposition provides a combinatorial description of the Yoneda product.

\begin{prop}
\label{pathProducts}
Let $p=p_0\cdots p_s$ and  $q=q_0\cdots q_n$ be admissible walks in $\Gamma(A)$. Then $\e_p\star \e_q = 0$ unless there exist walks $p'\sim p$ and $q'\sim q$ such that $q'$ is anchored and $q'_n\rightarrow p'_0$ is an edge in $\Gamma(A)$. In that case $\e_p\star \e_q = \e_{w}$ where $w\sim q'_0\cdots q'_np'_0\cdots p'_s$.
\end{prop}

\begin{proof}
 By Lemma \ref{pathEquivalence}(5) it suffices to consider the case where $q$ is anchored. For $i\ge 0$, let $f_i$ be defined as above. By definition of Yoneda composition product, $\e_p\star\e_q=\e_pf_{s+1}$. Let $r$ be any anchored walk of length $n+s+1$. If $r$ does not extend $q$, then $\e_pf_{s+1}(e_r)=0$. If $r$ extends $q$, then 
$$\e_pf_{s+1}(e_r)=
\begin{cases}
\e_p(\pi(m)e_{x_j}) & \text{ if $s=0$ and $r_{n+1}=m\tensor x_j$} \\
\e_p(e_{r'}) & \text{if $s>0$ is odd and $r_{n+2}\tensor r_{n+1}$ is essential}\\
\e_p(\pi(m')e_{r'}) & \text{if $s>0$ is even and $r_{n+2}\tensor r_{n+1}$ is essential}\\
0 & \text{else}\\
\end{cases}
$$
where $r'$ and $m'$ are defined as in (iii) above. Thus $\e_pf_{s+1}(e_r)=0$ unless $r$ extends $q$, $p\sim r'$, and, 
if $s$ is even, $r'_s=r_{n+s+1}$. The last condition implies $r'\sim r_{n+1}\cdots r_{n+s+1}$ when $s$ is even. (This equivalence always holds when $s$ is odd.) So if $\e_pf_{s+1}(e_r)\neq 0$, we have $\e_pf_{s+1}(e_r)=1$ and it follows that $\e_p\star\e_q=\e_w$ where $w=q_0\cdots q_n r_{n+1}\cdots r_{n+s+1}$. Since $p\sim r'\sim r_{n+1}\cdots r_{n+s+1}$, setting $p'=r_{n+1}\cdots r_{n+s+1}$ gives the desired result.

\end{proof}

We call a class $\alpha\in E^i(A)$ for $i>0$ \emph{decomposable} if $\alpha$ is in the subalgebra of $E(A)$ generated by $\bigoplus_{j<i} E^{j}(A)$. Otherwise we call $\alpha$ \emph{indecomposable}. Proposition \ref{pathProducts} illustrates a nice feature of our chosen $k$-basis for $E(A)$. 

\begin{cor}
\label{decompClasses}
If $w$ is an anchored walk of length $n$ in $\Gamma(A)$, then $\e_w$ is decomposable if and only if there exists an admissible walk $p=p_0\cdots p_m$ such that $w=w_0\cdots w_ip_0\cdots p_m$ for some $0\le i<n$.
\end{cor}

We note this implies the relations of $E(A)$ consist exclusively of monomials and binomials. (That $E(A)$ can be presented this way was also observed by C.\ Phan.)

We are nearly ready to give a combinatorial characterization of finite generation. We call an admissible walk $w$ \emph{decomposable} (resp. \emph{indecomposable}) if $\e_w$ is decomposable (resp. indecomposable) in $E(A)$. 

The following important fact is an application of the classical K\"{o}nig's Lemma (see \cite{HRM} p. 1046); its proof by induction is omitted.

\begin{lemma}
\label{Konig}
If $\Gamma(A)$ contains infinitely many indecomposable finite anchored walks, there exists an infinite anchored walk with infinitely many indecomposable finite prefixes.
\end{lemma}

Our main theorem characterizes infinite walks with infinitely many indecomposable prefixes. In the next section, we give a finite procedure for checking these conditions. If $p$ is a walk of length $n$ in $\Gamma(A)$, let $\widetilde{p}=p_1\cdots p_n$ be the walk $p$ with the initial edge deleted. Recall from Section \ref{CPSgraph} that if $e$ is an admissible edge in an infinite walk $p$, we call $e$ \emph{dense} in $p$ if $e$ has an admissible even-length extension in $p$.

\begin{thm}
\label{mainTheorem}
Let $A$ be a monomial $k$-algebra. The following are equivalent.
\begin{enumerate}
\item $E(A)$ is finitely generated.
\item Every infinite anchored walk in $\Gamma(A)$ has finitely many indecomposable prefixes.
\item For every infinite anchored walk $p$ in $\Gamma(A)$, $\widetilde{p}$ contains a dense edge or two admissible edges of opposite parity.
\end{enumerate} 
\end{thm}

Here ``opposite parity'' means the number of edges properly between the two admissible edges is even. See Section \ref{examples} for an illustration of the theorem.

\begin{proof} The equivalence of (1) and (2) follows from Lemma \ref{Konig}. We prove (2) and (3) are equivalent.

Let $p$ be an infinite anchored walk in $\Gamma(A)$. If $\widetilde{p}$ contains a dense edge $e$, then there exists an even-length extension $e\vdash q$ in $p$ such that $q$ is admissible.  
By Proposition \ref{pathExtensions}, every extension of $q$ is admissible, so by Corollary \ref{decompClasses}, $p$ has only finitely many indecomposable prefixes. 

By Proposition \ref{pathExtensions}, every odd-length extension of an admissible edge is admissible. Hence if $\widetilde{p}$ has admissible edges of opposite parity, $p$ has only finitely many indecomposable prefixes. 

Suppose instead that $\widetilde{p}$ has no dense edges and all admissible edges in $\widetilde{p}$ have the same parity. If $\widetilde{p}$ contains no admissible edges, then Corollary \ref{decompClasses} implies that every finite prefix of $p$ is indecomposable. If $\widetilde{p}$ contains an admissible edge, let $e=p_ip_{i+1}$ be the admissible edge with $i$ minimal. Since admissible edges have the same parity, for $n>0$ the admissible edges in $p_0\cdots p_{i+2n}$ have the form $p_{i+2j}p_{i+2j+1}$ for $0\le j<n$. Since a walk of the form $p_{i+2j}\cdots p_{i+2n}$ has even length and $\widetilde{p}$ contains no dense edges, $p_0\cdots p_{i+2n}$ is indecomposable for all $n>0$ by Corollary \ref{decompClasses}.

\end{proof}

\begin{rmk}
\label{firstEdgeDense}
If $w$ is an infinite walk and for $j>0$, $w_jw_{j+1}$ is a dense edge in $\widetilde{w}$, then any admissible edge $w_{j-2i}w_{j-2i+1}$, $0\le i\le \lfloor\frac{j}{2}\rfloor$ is also dense in $w$. This follows from the fact that $w_{j-2i}\cdots w_{j-1}$ is an odd-length extension of $w_{j-2i}w_{j-2i+1}$, Propositions \ref{pathExtensions}, \ref{pathBasis} and \ref{pathProducts}. Thus $\widetilde{w}$ contains a dense edge if and only if the first admissible edge in $\widetilde{w}$ is dense in $w$. It follows from the discussion in Section \ref{CPSgraph} that for the algebras $A$ and $B$ from Example \ref{ex1}, $E(A)$ is finitely generated and $E(B)$ is not.
\end{rmk}

\section{An Upper Bound for Checking Finite Generation}
\label{algorithm}

At first glance, verification of the conditions of Theorem \ref{mainTheorem} appears to require an infinite procedure, in general. The infinitude arises both from the number of infinite walks in $\Gamma(A)$ and the determination of edge density. In this section we establish an upper bound on the cohomological degree of an indecomposable element if $E(A)$ is finitely generated. 

For the first time, the distinction between ``path'' and ``walk'' is important. Let $L$ be the maximal length of an anchored path $p$ in $\Gamma(A)$ with $p_{L-1}p_L$ an admissible edge, and no edge $p_ip_{i+1}$ admissible for $0<i<L-1$.  Let $M$ be the size of the largest edge equivalence class, and let $\mathcal E$ be the number of edges of $\Gamma(A)$.

First we show that if $\text{gl.dim } A=\infty$ and $L=1$, $E(A)$ is not finitely generated. Thus in the sequel, we will focus our attention on the case $L>1$. In that case, the existence of an admissible edge $m_1\rightarrow m_2$ with $m_1\notin\G_0$ implies $M>1$.

\begin{lemma}
\label{L1}
Let $A$ be a monomial $k$-algebra with $\text{gl.dim } A=\infty$ and $L=1$. Then $E(A)$ is finitely generated if and only if every circuit in $\Gamma(A)$ contains a vertex in $\G_0$.
\end{lemma}

\begin{proof}
Every anchored walk is admissible. If every circuit in $\Gamma(A)$ contains a vertex in $\G_0$, then for any infinite walk $w$, the walk $\widetilde{w}$ contains a vertex in $\G_0$. Thus $\widetilde{w}$ contains a dense edge, and because $w$ was arbitrary, $E(A)$ is finitely generated by Theorem \ref{mainTheorem}.

 Since $\text{gl.dim } A=\infty$, the graph $\Gamma(A)$ contains a circuit. If $\Gamma(A)$ contains a circuit $C$ missing $\G_0$, let $p$ be an anchored path of length $n$ such that $p_n$ is in $C$. Since $L=1$, neither $\widetilde{p}$ nor $C$ contains an admissible edge. Let $q$ be the infinite extension of $p$ defined by repeatedly traversing $C$. Then $\widetilde{q}$ contains no admissible edge, hence every prefix of $q$ is indecomposable by Corollary \ref{decompClasses} and $E(A)$ is not finitely generated by Theorem \ref{mainTheorem}.
 
 \end{proof}

Next we establish an important upper bound.

\begin{lemma}
\label{bound}
Suppose $q$ is a walk of length $N>2\mathcal E (M-1)+L$ in $\Gamma(A)$. Assume $L>1$ and $\widetilde{q}$ contains admissible edge $q_{j}q_{j+1}$ for $0<j< L$. Let $p=q_{j}q_{j+1}\cdots q_{N}$ if $N-j$ is even and $p=q_{j}q_{j+1}\cdots q_{N-1}$ if $N-j$ is odd. Then
\begin{enumerate}
\item $p$ is admissible
\item for all $0\le i\le 2\mathcal E (M-1)$, we have $q_{2i+j}q_{2i+j+1}\sim p'_{2i}p'_{2i+1}$ where $p'\sim p$ and $p'$ is anchored.
\item Either 
\begin{enumerate}
\item $q_{2i+j}q_{2i+j+1}= p'_{2i}p'_{2i+1}$ for some $0\le i\le 2\mathcal E (M-1)$
or 
\item there exist $0\le c < d\le 2\mathcal E (M-1)$ such that 
$$q_{2c+j}q_{2c+j+1}=q_{2d+j}q_{2d+j+1}\quad \text{ and }\quad  p'_{2c}p'_{2c+1}=p'_{2d}p'_{2d+1}$$
\end{enumerate}
\end{enumerate}
\end{lemma}

\begin{proof}
Statement (1) is immediate from Proposition \ref{pathExtensions} since the length of $p$ is odd. Statement (2) then follows from Lemma \ref{pathEquivalence}.
To prove (3), let $\mathbb E$ denote the set of edges of $\Gamma(A)$ and let $$S=\{(e_1, e_2)\in \mathbb E\times \mathbb E : e_1\sim e_2, e_1\neq e_2\}$$ Then $|S|\le \mathcal E (M-1)$. Since the walks $p$ and $p'$ consist of at least $2\mathcal E (M-1)$ edges of $\Gamma(A)$, either one of the pairs  of equivalent edges
$$(q_{2i+j}q_{2i+j+1}, p'_{2i}p'_{2i+1})\qquad 0\le i\le \mathcal E (M-1)$$ is not in $S$, in which case (a) holds, or some element of $S$ appears twice, in which case (b) holds.

\end{proof}

\begin{thm}
\label{finiteCheck}
Let $A$ be a monomial $k$-algebra with $\text{gl.dim } A=\infty$ and $L>1$. Let $N$ be the smallest even integer greater than or equal to $2\mathcal E (M-1)+L+1$. The Yoneda algebra $E(A)$ is finitely generated if and only if every anchored walk $q$ of length $N$ or $N+1$ is decomposable. 
\end{thm}

Since $M-1$ and $L$ are at most $\mathcal E$, we obtain the weaker, but more easily stated bound of $2\mathcal E^2+\mathcal E+1$ mentioned in the Introduction.

\begin{proof}

Suppose every anchored walk of length $N$ or $N+1$ is decomposable. Let $q$ be any anchored walk of length $N+1$. Then $q$ and $q'=q_0\cdots q_{N}$ are both decomposable. By Corollary \ref{decompClasses}, $\widetilde{q'}$ contains an admissible edge $q_iq_{i+1}$. By Proposition \ref{pathExtensions}, every odd length extension of $q_iq_{i+1}$ is admissible. This can account for the decomposability of only one of $q$ and $q'$. Since both are decomposable, either $q$ contains an admissible edge whose parity is opposite $q_iq_{i+1}$ or an even length extension of $q_iq_{i+1}$ is admissible, making $q_iq_{i+1}$ dense in any infinite walk with prefix $q$. Since $q$ was arbitrary, $E(A)$ is finitely generated by Theorem \ref{mainTheorem}.

Conversely, suppose $\Gamma(A)$ contains an indecomposable anchored walk $q$ of length $N$ or $N+1$. We will construct an infinite anchored walk $w$ in $\Gamma(A)$ in which all admissible edges have the same parity, but none are dense in $w$. That $E(A)$ is not finitely generated will then follow from Theorem \ref{mainTheorem}. 

By the discussion preceeding Lemma \ref{L1}, we have $M>1$, hence $\widetilde{q}$ contains a circuit. If $\widetilde{q}$ contains no admissible edge, or if the first admissible edge of $\widetilde{q}$ follows a circuit in $\widetilde{q}$, we can construct an infinite walk in $\Gamma(A)$ in which every prefix is indecomposable as in the proof of Lemma \ref{L1}. Otherwise, let $0<j<L$ be minimal such that $q_jq_{j+1}$ is an admissible edge of $\widetilde{q}$. 

Since $q$ is indecomposable, Proposition \ref{pathExtensions} implies the length of $q$ and the index $j$ must have opposite parity. We consider only the case where $q$ has length $N$, the other case being identical after the obvious necessary index shift.  Let $p=q_j\cdots q_{N-1}$. By Lemma \ref{bound}(1), $p$ is admissible, so let $p'$ be the unique anchored walk equivalent to $p$. 

The walk $q_j\cdots q_N$ is not admissible, so by Lemma \ref{density} we must have $q_{2i+j}\neq p'_{2i}$ for all $0\le i\le 2\mathcal E (M-1)$. Therefore, we have
 $q_{2i+j}q_{2i+j+1}\neq p'_{2i}p'_{2i+1}$ for all $0\le i\le 2\mathcal E (M-1)$.
By Lemma \ref{bound}(3), there exist $0\le c<d\le 2\mathcal E (M-1)$ such that $q_{2c+j}q_{2c+j+1}=q_{2d+j}q_{2d+j+1}$
and $p'_{2c}p'_{2c+1}=p'_{2d}p'_{2d+1}$. 

Let $z=q_{2c+j}\cdots q_{2d+j-1}$, let $z'=p'_{2c}\cdots p'_{2d-1}$ and let $w$ be the infinite walk $$q_0\cdots q_{2c+j-1}zzz\cdots$$
Since $q_{2c+j}=q_{2d+j}$ and $q_{2d+j-1}\rightarrow q_{2d+j}$ is an edge in $\Gamma(A)$, the walk $w$ is indeed well-defined. Likewise, the walk $$w'=p'_0\cdots p'_{2d-1}z'z'z'\cdots$$ is well-defined. Since all admissible edges of $\widetilde{q}$ have the same parity as $q_jq_{j+1}$, the same is true for $\widetilde{w}$. Moreover, every admissible extension of $q_jq_{j+1}$ in $w$ is a prefix of $w'$. Since  $q_{2i+j}\neq p'_{2i}$ for all $0\le i\le d$ as noted above, the edge $q_jq_{j+1}$ is not dense in $w$ by Lemma \ref{density}. By Remark \ref{firstEdgeDense}, $\widetilde{w}$ contains no dense edges. Therefore, $E(A)$ is not finitely generated by Theorem \ref{mainTheorem}.

\end{proof}

In many cases, one can determine if $E(A)$ is finitely generated well before the upper bound above. 
Indeed if $\GKdim\ E(A)=1$, there are finitely many infinite anchored walks. If $d=\GKdim\ E(A)<\infty$, it is easy to describe a recursive procedure: 
\begin{enumerate}
\item Analyze the (finite number of) subgraphs of $\Gamma(A)$ with at most $d-1$ distinct circuits (ignoring multiplicity) in any walk.
\item If no anchored walk with infinitely many indecomposable prefixes is found, let $M$ be the maximal multiplicity of a circuit in an indecomposable walk. Analyze the (finite number of) infinite walks $w$ containing $d$ distinct circuits (ignoring multiplicity) such that the first $d-1$ circuits of $w$ occur with multiplicity $\le M$.
\end{enumerate}

\section{The Noetherian Property}
\label{sectionNoeth}

Green et.\ al.\ \cite{GSSZ} observed that if $A$ is a monomial quadratic algebra, it is possible to determine if $E(A)$ is Noetherian by considering its Ufnarovski relation graph. As noted in Section \ref{CPSgraph}, if $A$ is a monomial quadratic algebra, then $\Gamma(A)$ is precisely the Ufnarovski graph of the Koszul dual $A^!\isom E(A)$ with edge orientations reversed. In this section we prove an analog of Green et.\ al.'s ``Noetherianity'' theorem (Theorem 5.4 of \cite{GSSZ}) holds for $\Gamma(A)$.

The following Lemma illustrates an important difference between quadratic monomial algebras and monomial algebras with defining relations in higher degrees. The Lemma also conveys the sense in which Theorem \ref{Noeth} below generalizes the result in \cite{GSSZ}. 

\begin{lemma}
Let $A$ be a monomial $k$-algebra such that the defining ideal of $A$ is generated by quadratic monomials. Then $\G=\G_0$ and every edge of $\Gamma(A)$ is admissible.
\end{lemma}

\begin{proof}
Since the minimal generators of $I$ are quadratic, for any generator $x_j$, $\A_{x_j}$ consists of linear monomials. Thus $\G_1\subset\G_0$ and $\G=\G_0$. It follows (see Remark \ref{G0edge}) that every edge of $\Gamma(A)$ is admissible. 

\end{proof}

For our discussion of the Noetherian property, we discard the assumption that ideals in a graded $k$-algebra are generated by homogeneous elements of degrees $\ge 2$.

To establish the main theorem of this section in the left Noetherian case, we filter a left ideal by defining a total order on the path basis of Proposition \ref{pathBasis}. We invoke this total order only when $\Gamma(A)$ has the property that every vertex lying on an oriented circuit has out-degree 1. To handle the right Noetherian case, one first defines the analogous total order under the assumption that every vertex on an oriented circuit has in-degree 1. In the interest of brevity, we provide details only for the left Noetherian case. We define the order in several steps. 

We first fix a total ordering of the $s$ circuits of $\Gamma(A)$: $C_1<C_2<\cdots< C_s$. An \emph{in-path} $p$ for a circuit $C_i$ is an anchored path $p$ with the final vertex of $p$ on $C_i$ and no other vertex of $p$ on $C_i$. The set of in-paths to a particular circuit is finite, so we fix a total ordering on each set of in-paths. Of the maximal paths in $\Gamma(A)$, finitely many terminate on no circuit. We fix a total ordering on these paths as well and define them to be less than any in-path.

If $p$ and $q$ are in-paths of lengths $n$ and $m$ for circuits $C_i$ and $C_j$ respectively, we define $p<q$ if $n<m$ or $n=m$ and $i<j$ or $n=m$ and $i=j$ and $p<q$ in the fixed ordering on in-paths of $C_i$.

If $w$ is any anchored walk in $\Gamma(A)$, then there exists a unique path $\overline{w}$ such that exactly one of the following holds:
\begin{itemize}
\item $\overline{w}$ is an in-path terminating on $C_i$ with $i$ minimal and $\overline{w}$ extends $w$ or
\item $\overline{w}$ is an in-path terminating on $C_i$ and is a proper prefix of $w$  or
\item $\overline{w}$ is a maximal extension of $w$ terminating on no circuit
\end{itemize}

If $p$ and $q$ are anchored walks of lengths $n$ and $m$ respectively, we define $\e_p<\e_q$ if $n<m$ or if $n=m$ and  $\overline{p}<\overline{q}$.

\begin{thm}
\label{Noeth}
For a monomial $k$-algebra $A$, the Yoneda algebra $E(A)$ is left (resp.\ right) Noetherian if and only if 
\begin{enumerate}
\item every vertex of $\Gamma(A)$ lying on an oriented circuit has out-degree (resp. in-degree) 1, and
\item every edge of every oriented circuit is admissible.
\end{enumerate}
\end{thm}

\begin{proof}
%As noted above, it suffices to consider the case $\GKdim(E(A))<\infty$.
If $\Gamma(A)$ contains no circuit, $E(A)$ is finite dimensional, hence Noetherian, by Corollary \ref{GKfacts}. Assume $\Gamma(A)$ contains a circuit. 

First suppose $\Gamma(A)$ satisfies conditions (1) and (2).  Let $J$ be any left ideal of $E(A)$. We claim $J$ is finitely generated.

Order the path basis as described above. For any class $\e\in E(A)$, let $h(\e)$ be the largest basis element appearing with nonzero coefficient when $\e$ is expressed in the path basis. Let $F^{\bullet}$ be the natural filtration on $J$ inherited from the cohomology grading on $E(A)$. For $n>0$ let $\mathcal L_n$ be the left ideal generated by $\{h(\e) : \e\in F^nJ\} $. Let $\mathcal L=\bigcup_n \mathcal L_n$.

Conditions (1) and (2) guarantee the existence of a largest integer $d$ such that the final edge in a path $p$ of length $d$ is not admissible. Then by Corollary \ref{decompClasses}, $\e_q$ is decomposable for any anchored walk $q$ of length $>d$. It follows that $\mathcal L/\mathcal L_d$ is finitely generated as a left ideal (if not, one could find anchored walks $p$ and $q$ with $p\vdash q$ and $\e_p$ and $\e_q$ algebraically independent), hence the ascending chain of left ideals $\mathcal L_1\subset \mathcal L_2\subset \cdots$ stabilizes. The fact that $J$ is finitely generated then follows by the standard Hilbert Basis argument. 

Conversely, let $C=c_0\cdots c_n$ be a circuit of length $n$ in $\Gamma(A)$. First suppose vertex $c_i$ has out-degree $>1$, where $0\le i<n$. Let $v\neq c_{i+1}$ be a vertex such that $c_i\rightarrow v$ is an edge in $\Gamma(A)$. Let $p$ be any anchored path of length $m$ in $\Gamma(A)$ such that $p_m=c_{n-1}$. For $\ell\ge 0$, define $q_{\ell} = pC^{\ell}c_0\cdots c_iv$ where $C^{\ell}$ indicates the circuit $C$ is traversed $\ell$ times. Let $J$ be the left ideal of $E(A)$ generated by $\{\e_{q_{\ell}} : \ell \ge 0\}$. We claim that $J$ is not finitely generated.

If $J$ is finitely generated, there exists $L>0$ such that $\e_{q_{\ell}}$ is in the left ideal generated by $\e_{q_0}, \ldots, \e_{q_L}$ for all $\ell>L$. Fix $\ell_0>L$. Then by Proposition \ref{pathBasis} and Corollary \ref{decompClasses}, there exists a walk $w$ and an index $0\le d\le L$ such that $\e_{q_{\ell}}=\e_w\star \e_{q_d}$ and $q_{\ell}\sim q_{d}w$.
Since $v\neq c_{i+1}$ and since $pC^dc_0\cdots c_i$ is a prefix of $q_{\ell}$, Lemma \ref{pathEquivalence}(2) implies $q_d$ is an even-length walk. But by Lemma \ref{pathEquivalence}(3) and (4), $w_0\tensor v=c_{i+2}\tensor c_{i+1}$ (where, if $i=n-1$, $c_{i+2}=c_1$) and $\deg(v)=\deg(c_{i+1})$, implying $v=c_{i+1}$, a contradiction. Thus if $\Gamma(A)$ contains a vertex of out-degree $>1$ lying on a circuit,  $E(A)$ is not left Noetherian.

Suppose instead that every vertex of $C$ has out-degree 1 and $C$ contains an edge $c_j\rightarrow c_{j+1}$ which is not admissible. Let $K$ be the left ideal of $E(A)$ generated by $\e_{q_i}$ for $i\ge 0$ where $q_i=pC^ic_0\cdots c_{j-1}$ (if $j=0$, then since $c_0=c_n$ we take $q_i=pC^ic_0\cdots c_{n-1}$). Since $c_jc_{j+1}$ is not admissible, and since $c_j$ is the only successor of $c_{j-1}$ in $\Gamma(A)$, it follows from Lemma \ref{pathEquivalence}(1) and Proposition \ref{pathProducts} that $\e_w\star \e_{q_i}=0$ for any admissible walk $w$. Thus $K$ is an infinitely-generated trivial left ideal and $E(A)$ is not left Noetherian.

We omit the analogous proof for the right Noetherian case.

\end{proof}

For the algebras $A$ and $B$ of Example \ref{ex1}, $E(A)$ and $E(B)$ are neither left nor right Noetherian. Comparing our graph-theoretic characterizations of GK dimension and the Noetherian property, we have the following immediate corollary.

\begin{cor}
Let $A$ be a monomial $k$-algebra. If $E(A)$ is left or right Noetherian, then $\GKdim E(A)\le 1$. If  $E(A)$ is Noetherian then $\Gamma(A)$ consists of finitely many disjoint circuits and paths.
\end{cor}

\section{Examples}
\label{examples}

The following example suggests that the $\GKdim(E(A))=\infty$ case can be quite complicated; edges that are dense in one infinite walk need not be dense in another.

\begin{ex}
\label{infGK}
Let $S=k\la w, x, y, z, W, X, Y, Z, p,q\ra$ be a free algebra and let $I$ be the ideal generated by 
$$\begin{array}{cccc}
pqwxyz & wxyzp & xyzpqwx & YZpqwx \\ 
 pqWXYZ& WXYZp   & XYZpqWX & yzpqWX\\ 
\end{array}$$
Let $A$ be the factor algebra $A=S/I$. The graph $\Gamma(A)$ has two components and is shown in Figure 1 below.  Admissible edges are indicated by solid arrows; dashed arrows are non-admissible edges. Vertex $pq$ is common to two oriented circuits, so $\GKdim(E(A))=\infty$. There are many infinite walks in $\Gamma(A)$. The walk
$$p\rightarrow wxyz\rightarrow pq\rightarrow wxyz\rightarrow pq\rightarrow \cdots$$
is anchored and $wxyz\rightarrow pq$ is dense in this walk since the even-length walk
$$wxyz\rightarrow pq\rightarrow wxyz\rightarrow pq\rightarrow wxyz$$ is equivalent to $$z\rightarrow pqwxy\rightarrow xyz\rightarrow pqw\rightarrow wxyz$$
However, the walk
$$p\rightarrow wxyz\rightarrow pq\rightarrow WXYZ\rightarrow pq\rightarrow wxyz\rightarrow pq\rightarrow WXYZ\rightarrow \cdots$$
contains no dense edge. To see this, observe that the equivalent anchored walks corresponding to odd-length admissible extensions of $wxyz$ (and likewise of $WXYZ$) terminate on the circuit
$$yz\rightarrow pqwx\rightarrow YZ\rightarrow pqWX\rightarrow yz$$
It follows that no even-length extension of $wxyz\rightarrow pq$ is admissible because condition (2) of Lemma \ref{pathEquivalence} cannot be satisfied. By Theorem \ref{mainTheorem}, $E(A)$ is not finitely generated.

\end{ex}

\begin{figure}
\label{figure}
$$
\xymatrix{
 xyzpqw\ar@{-->}[r] & w &   &   xyzpq\ar@{-->}[ll]\\
 x\ar@{->}[d]\ar@{->}[u] &  & YZpq\ar@{-->}[ld]   \\
 YZpqw\ar@{-->}[r] & WX\ar@{->}[dd]\ar@{->}[rd]   & & wx\ar@{->}[lu]\ar@{->}[uu] & yzpqW\ar@{-->}[l]  \\ 
  &  & yzpq\ar@{-->}[ru] &  & X\ar@{->}[u]\ar@{->}[d]   \\
  &   XYZpq\ar@{-->}[rr]&  &W & XYZpqW\ar@{-->}[l]   \\
}
$$

$$
\xymatrix{
    &       & z\ar@{->}[r]  &  pqwxy\ar@{-->}[rd]\ar@{-->}[ld] & \\
    &       & YZ\ar@{->}[ld]   &        & xyz\ar@{->}[ddddd] \\
    & pqWX\ar@{->}[dr]\ar@{->}[dl]  &      & pqwx \ar@{->}[ru]\ar@{->}[lu]   & \\
XYZ\ar@{->}[ddd] &       & yz\ar@{->}[ru]   & & \\
    & pqWXY\ar@{-->}[ur]\ar@{-->}[ul] & Z\ar@{->}[l]    & & \\
               &       & WXYZ\ar@/^/[d] & & \\
pqW\ar@{-->}[rru]\ar@{-->}[rrd] &   p\ar@{->}[ru]\ar@{->}[rd]   & pq\ar@/^/@{-->}[u] \ar@/^/@{-->}[d] & & pqw\ar@{-->}[llu]\ar@{-->}[lld]\\
    &       & wxyz\ar@/^/[u] & & \\
}
$$
\caption{The graph $\Gamma(A)$ for Example \ref{infGK}.}
\end{figure}
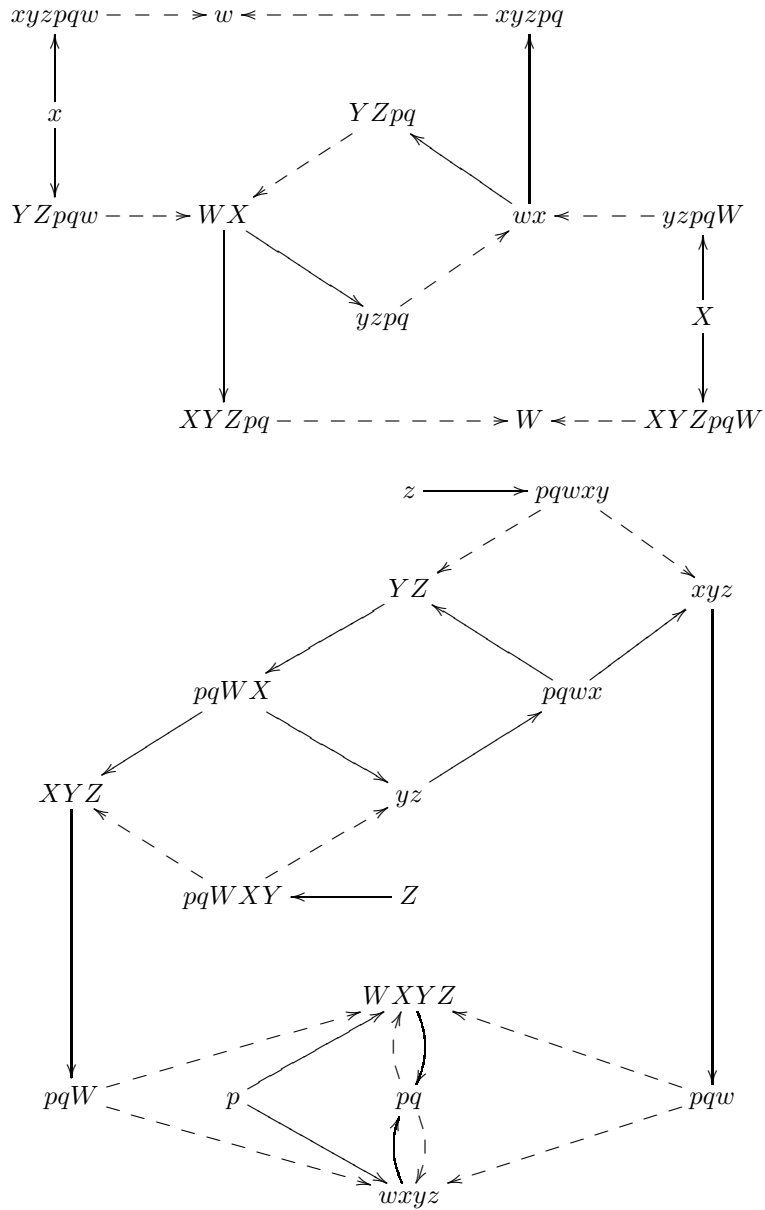

\begin{ex}
Let $A=k\la x, y\ra/\la x^3-x^2y, xy^2, y^3\ra$ and observe $x^4=0$ in $A$. The degree-lexicographic ordering on monomials in $k\la x, y\ra$ with $x<y$ yields the associated monomial algebra $A'= k\la x, y\ra/\la x^2y, xy^2, y^3, x^4\ra$. Although $\dim E^i(A)\le\dim E^i(A')$ for all $i$, one can check that equality does not always hold. The graph $\Gamma(A')$ is shown below. By Corollary \ref{GKfacts}, we have $\GKdim(E(A'))=2$. It follows that $\GKdim(E(A))\le 2$. 
$$
\xymatrix{
x^3 \ar@/^/@{->}[r] & x \ar@/^/@{->}[l] & y^2\ar@{->}[l] \ar@/^/@{->}[d] & \\
    & xy\ar@{->}[u] & y\ar@{->}[l]\ar@{->}[r] \ar@/^/@{->}[u] & x^2\ar@(ur,dr)[] \\
    & & \Gamma(A') &\\
}
$$
We leave to the reader the straightforward verification that $\GKdim(E(A))>1$, hence $\GKdim(E(A))=2$ by Bergman's gap theorem.

In many cases of interest, knowing $\GKdim(E(A'))$ provides little or no informaiton about $\GKdim(E(A))$. Consider the algebra 
$$A=\dfrac{k\la x,y,z\ra}{\la xy - z^2, zx - y^2, yz - x^2\ra}$$ The algebra $A$ is a 3-dimensional Sklyanin algebra, hence $\GKdim(E(A))=0$. Using lexicographic ordering with $z>y>x$, the associated monomial algebra of $A$ is
$$A' = \dfrac{k\la x, y, z\ra}{\la z^2, zx, yz, y^3, zy^2, yxy, yx^3, y^2x^2, zyx^2\ra}$$ Constructing the CPS graph of $A'$ reveals $\GKdim(E(A'))=\infty$.

\end{ex}

\bibliographystyle{amsplain}
\bibliography{bibliog2}

\end{document}